\documentclass[11pt]{article}
  \usepackage{a4wide}
  \usepackage{amsmath}
  \usepackage{amssymb}
  \usepackage{latexsym}
  \usepackage[pdftex]{graphicx}
  \usepackage{subcaption}
  \usepackage{color}
  \usepackage[mathscr]{euscript}
  \usepackage{float}
  \usepackage{url}

  \newenvironment{proof}{\vspace{1ex}\noindent{\bf Proof.}}{\hspace*{\fill}$\blacksquare$\vspace{1ex}}
  \newenvironment{proofof}[1]{\vspace{1ex}\noindent{\bf Proof of #1.}}{\hspace*{\fill}$\blacksquare$\vspace{1ex}}

  \newtheorem{theorem}{Theorem} 
  \newtheorem{lemma} [theorem] {Lemma}
  \newtheorem{proposition} [theorem] {Proposition}
 
\newcommand{\Pcal}[0]{\ensuremath{{\mathcal P}}}

\newcommand{\Scal}[0]{\ensuremath{{\mathcal S}}}
\newcommand{\Tcal}[0]{\ensuremath{{\mathcal T}}}

\newcommand{\Xcal}[0]{\ensuremath{{\mathcal X}}}
\newcommand{\Ycal}[0]{\ensuremath{{\mathcal Y}}}
\newcommand{\Zcal}[0]{\ensuremath{{\mathcal Z}}}
\newcommand{\eR}[0]{\ensuremath{\mathbb R}}
\newcommand{\Haa}[0]{\ensuremath{\mathbb H}}

\newcommand{\eN}[0]{\ensuremath{ \mathbb N}}
\newcommand{\Zed}[0]{\ensuremath{ \mathbb Z}}

\newcommand{\Dee}[0]{\ensuremath{\mathbb D}}

\newcommand{\norm}[1]{\ensuremath{\|#1\|}}

\newcommand{\Pee}[0]{\ensuremath{{\mathbb P}}}

\newcommand{\isd}[0]{\hspace{.2ex} \raisebox{-.1ex}{$=$} \hspace{-1.5ex} 
\raisebox{1ex}{{$\scriptstyle d$}} \hspace{.8ex} }

\DeclareMathOperator{\dist}{dist}

\DeclareMathOperator{\diam}{diam}

\DeclareMathOperator{\cross}{cross}

\DeclareMathOperator{\area}{area}

\definecolor{orange}{RGB}{255,127,0}
\definecolor{pink}{RGB}{255,150,150}

\DeclareMathOperator{\arcsinh}{arcsinh}

\newcommand{\nearby}{\text{\textup{local}}}
\DeclareMathOperator{\closed}{closed}

\newcommand{\ballR}[0]{\ensuremath{B_{\eR^2}}}
\newcommand{\ballH}[0]{\ensuremath{B_{\Haa^2}}}
\newcommand{\cellR}[0]{\ensuremath{C_{\eR^2}}}
\newcommand{\cellH}[0]{\ensuremath{C_{\Haa^2}}}

\newcommand{\distH}[0]{\ensuremath{\dist_{\Haa^2}}}
\newcommand{\areaH}[0]{\ensuremath{\area_{\Haa^2}}}
\newcommand{\areaR}[0]{\ensuremath{\area_{\eR^2}}}
\newcommand{\Zcalb}[0]{\ensuremath{\Zcal_{\text{b}}}}
\newcommand{\Zcalw}[0]{\ensuremath{\Zcal_{\text{w}}}}
\newcommand{\Zcalbtil}[0]{\ensuremath{\tilde{\Zcalb}}}
\newcommand{\Zcalwtil}[0]{\ensuremath{\tilde{\Zcalw}}}
\newcommand{\Zcaltil}[0]{\ensuremath{\tilde{\Zcal}}}
\newcommand{\pnew}[0]{\ensuremath{p_{\text{new}}}}
\newcommand{\diamR}[0]{\ensuremath{\diam_{\eR^2}}}

\begin{document}

\title{The critical probability for Voronoi percolation \\ in the hyperbolic plane tends to $1/2$}

\author{%
Benjamin T.~Hansen\thanks{Bernoulli Institute, 
Groningen University, The Netherlands. 
E-mail: {\tt b.t.hansen@rug.nl}.  Supported by the Netherlands Organisation for Scientific Research (NWO) 
under project no. 639.032.529.}
\and 
Tobias M\"uller\thanks{Bernoulli Institute, 
Groningen University, The Netherlands. 
E-mail: {\tt tobias.muller@rug.nl}. Supported in part by the Netherlands Organisation for Scientific Research (NWO) 
under project nos 612.001.409 and 639.032.529.}%
}

\date{\today}

\maketitle

\begin{abstract}
We consider percolation on the Voronoi tessellation generated by a homogeneous Poisson point process on the hyperbolic plane.  
We show that the critical probability for the existence of an infinite cluster tends to $1/2$ as the intensity of the 
Poisson process tends to infinity.
This confirms a conjecture of Benjamini and Schramm~\cite{benjamini2001percolation}.
\end{abstract}

\section{Introduction and statement of result}

In this paper, we will consider percolation on the Voronoi tessellation generated by a homogeneous Poisson point process 
on the hyperbolic plane $\Haa^2$. That is, with each point of a constant intensity Poisson process on $\Haa^2$ we associate 
its Voronoi cell -- which is the set of all points of the hyperbolic plane that are closer to it than to any other point 
of the Poisson process -- and we colour each cell black with probability $p$ and white with probability $1-p$, independently 
of the colours of all other cells.
We refer the reader to Section~\ref{sec:preliminaries} for detailed definitions and some background on the hyperbolic plane, 
hyperbolic Poisson point processes and their Voronoi tessellations.
Figure~\ref{fig:perc} shows a computer simulation of Voronoi percolation in the hyperbolic plane, rendered
in the Poincar\'e disk representation of $\Haa^2$. (Note that in all depictions of hyperbolic Voronoi percolation 
in the paper the black cells are rendered light blue to aid visibility.)
\begin{figure}[h!]
	\center{\includegraphics[width=.8\textwidth ]{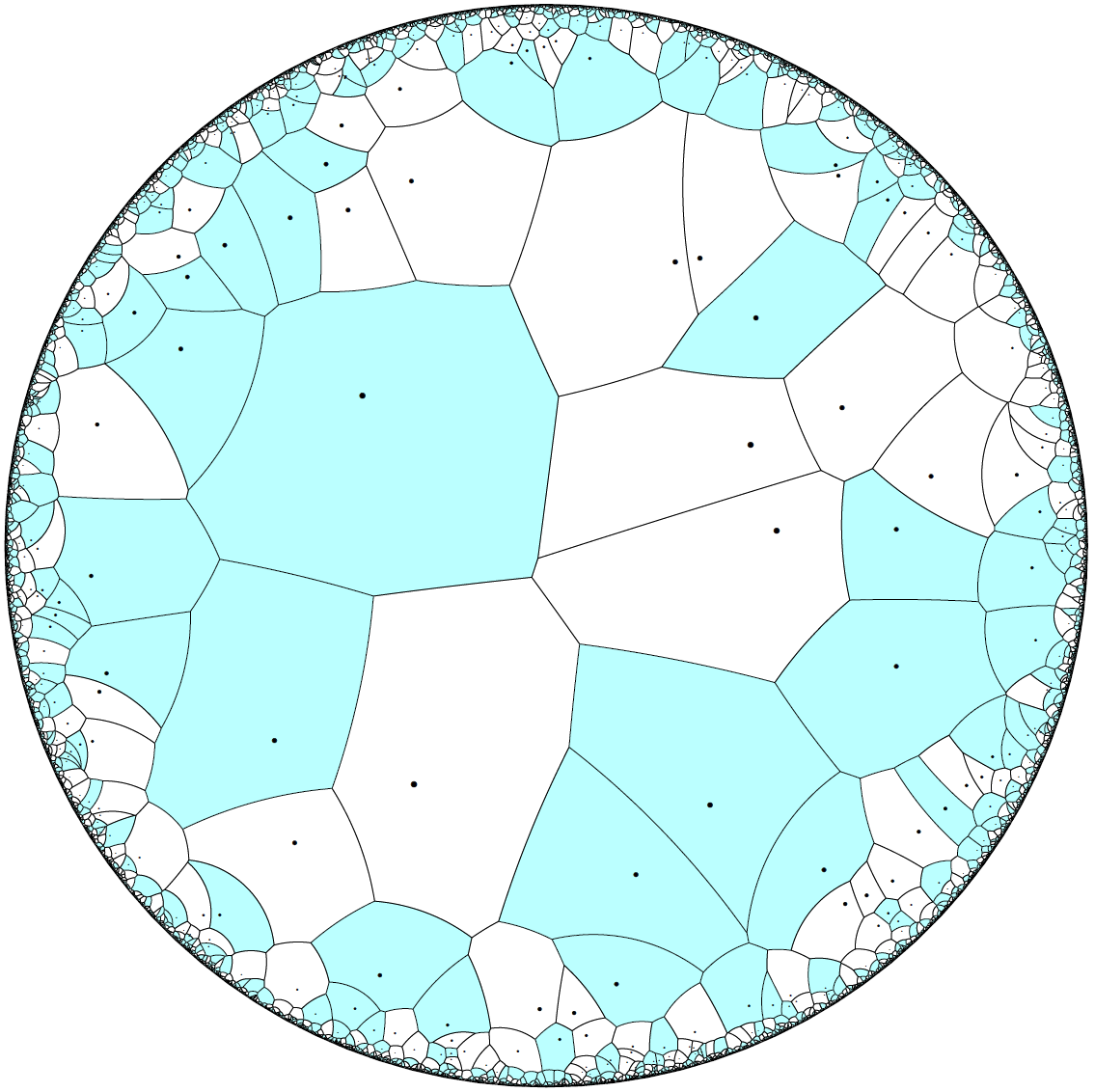} }
	\caption{\label{fig:perc} Computer simulation of Voronoi percolation in the hyperbolic plane, shown 
	in the Poincar\'e disk representation of $\Haa^2$. ($p=1/2, \lambda=1$)  }
\end{figure}%
We say that \emph{percolation} occurs if there is an infinite connected cluster of black cells. 
For each intensity $\lambda>0$ of the underlying Poisson process, the \emph{critical probability} is defined as 

$$ p_c(\lambda):=\inf\{p:\Pee_{\lambda,p}(\text{percolation})>0\}. $$

In an influential paper, Benjamini and Schramm~\cite{benjamini2001percolation} showed that $0<p_c(\lambda)<1/2$ for all $\lambda$
and they conjectured that $p_c(\lambda)$ tends to $1/2$ as $\lambda\to\infty$.
Here we prove their conjecture.

\begin{theorem}\label{thm:main}
$\displaystyle \lim_{\lambda\to\infty}p_c(\lambda)=1/2$.
\end{theorem}

The results of Benjamini and Schramm~\cite{benjamini2001percolation} include that $p_c(\lambda)$ tends to
zero as $\lambda\searrow 0$ and that for any $p\in (p_c,1-p_c)$ there are infinitely many infinite black clusters 
almost surely, while for $p\geq 1-p_c$ there is almost surely exactly one infinite black cluster.
These results demonstrate profound differences with Voronoi percolation on the ordinary, Euclidean plane.
In that case, the critical probability was shown 
to equal $1/2$ in the work of Zvavitch~\cite{Zvavitch} and Bollob\'as and Riordan~\cite{bollobas2006critical}, no matter
the intensity of the underlying Poisson process. (That the precise value of the intensity is irrelevant in the Euclidean setting 
is easily seen by considering the effect of dilations $x\mapsto\rho x$ on a homogeneous Poisson process and its Voronoi tessellation.)
What is more, in the Euclidean case there is almost surely at most one infinite cluster for any value of $p$ and $\lambda$. 

Some intuition as to why one might expect that the large $\lambda$ limit of the critical value in the hyperbolic plane
might coincide with the critical value in the Euclidean plane is given by the following observation.
As we ``zoom in'' around any given point on $\Haa^2$, the geometry starts to look more and more like the geometry of the Euclidean plane; 
and as the intensity $\lambda\to\infty$ the points of the Poisson process get packed closer and closer together.
In other words, as $\lambda\to\infty$, it becomes increasingly difficult to distinguish the ``local picture'' from ordinary, Euclidean 
Voronoi percolation. 
Of course the same is not necessarily true for global characteristics of the model (for instance, the simple random 
walk on the Voronoi tessellation is recurrent in the Euclidean case~\cite{Louigi} and transient in the hyperbolic 
case~\cite{Elliot}, irrespective of the value of the intensity parameter $\lambda$).
So additional ideas are needed for the proof of Theorem~\ref{thm:main}.

As in other two-dimensional percolation models, a central role in the proof by Bollob\'as and Riordan~\cite{bollobas2006critical}
that the Euclidean Voronoi percolation model almost surely has an infinite black cluster for every $p>1/2$
is played by crossings of rectangles. (See Section~\ref{sec:crosseucl} for the formal definition.) 
It follows from their work that for any $p>1/2$ and any fixed rectangle, the crossing probability tends to one as
$\lambda\to\infty$ (Bollob\'as and Riordan did not state precisely this, as they did not need this statement 
for their proof, but it follows from their work as we point out in more detail in Section~\ref{sec:crosseucl}). 
An ingenious and technically involved argument in their proof establishes that at $p=1/2$ the crossing probabilities
do not tend to zero. They then apply tools from discrete Fourier analysis to show the
crossing probabilities are ``boosted'' to close to one for $p>1/2$. 
Tassion~\cite{Tassion16} later improved over the $p=1/2$ part of their proof, giving a shorter argument showing the stronger statement 
that when $p=1/2$ the crossing probabilities are in fact bounded away from zero. 
This opened the way for a more detailed picture of the behaviour of the model at the critical value $p=1/2$
(see e.g.~\cite{AhlbergEtAl16, Vanneuville19}).
We will however only make use of the contributions of Bollob\'as and Riordan~\cite{bollobas2006critical} in 
our proof of Theorem~\ref{thm:main}.

We remark that here in the introduction and in the rest of the paper we focus on mentioning works directly relevant to percolation on 
Voronoi tesselations. 
Percolation theory is of course a broader and older subject, and some of the tools from the literature we'll rely on (and cite) are adaptations 
of earlier results that were geared towards percolation on lattices. 
The monograph~\cite{bollobas2006percolation} provides an introduction to percolation theory.

\vspace{1ex}

{\bf Sketch of the main ideas used in the proof.}
A crucial ingredient in our proof is Lemma~\ref{lem:cross} below, stating that in the hyperbolic Voronoi percolation model
with $p>1/2$, any fixed rectangle $R$ 
has a black crossing with probability tending to one 
as $\lambda\to\infty$. (Here ``rectangle'' refers to how $R$ appears in the Poincar\'e disk model of $\Haa^2$.)
We derive this from the results on crossings in the Euclidean case mentioned above. 
First, we argue that it is sufficient to prove the statement only for small enough rectangles. Then we employ a coupling, 
provided by Lemma~\ref{lem:coupling} below, to Euclidean Voronoi percolation with a slightly smaller value of $p$ and an 
appropriate value of the intensity parameter. 
This coupling has the property that if $B \subseteq R$ is the region inside our sufficiently small rectangle 
coloured black by the Voronoi tessellation of the hyperbolic Poisson point process and $\tilde{B} \subseteq R$ the region
coloured black by the Voronoi tessellation of the Euclidean Poisson point process, then $\tilde{B} \subseteq B$ with 
probability tending to one as $\lambda\to\infty$. So if there is a crossing of $R$ 
in the Euclidean Poisson point process, then there will be one in the hyperbolic Poisson process. 

We now sketch how we use Lemma~\ref{lem:cross} to prove Theorem~\ref{thm:main}.
Since Benjamini and Schramm~\cite{benjamini2001percolation} have already shown that $p_c(\lambda)<1/2$ for all
$\lambda$, it suffices to show that for every $p<1/2$ there exists a $\lambda_0 = \lambda_0(p)$ such that 
$\Pee_{\lambda,p}(\text{$\exists$ infinite black cluster}) = 0$ for all $\lambda > \lambda_0$.
Equivalently, switching the roles of black and white, it suffices to show that for every $p>1/2$ we have
$\Pee_{\lambda,p}( \text{$\exists$ infinite white cluster} ) = 0$ for all sufficiently large $\lambda$.
We will define a dependent site percolation model on a certain, fixed triangulation $\Tcal$ of $\Haa^2$, with the property 
that the existence of an infinite white cluster in the hyperbolic Voronoi tessellation implies the existence
of an infinite open cluster in the dependent percolation model.
Lemma~\ref{lem:nearby} provides, for each bounded set $A \subseteq \Haa^2$ each $\delta>0$, an event $\nearby(A,\delta)$
that holds with probability tending to one as $\lambda\to\infty$, and 
with the property that if it occurs then the colouring of $A$ depends only on the part of the Poisson process
within distance $\delta$ of $A$.
For each triangle of $\Tcal$, we place six thin rectangles around it in such a way that 
if they each have a crossing, then there will be a black, continuous, closed curve separating the triangle from infinity.
We declare a triangle of $\Tcal$ {\em closed} if the six rectangles each have a crossing and in addition 
the event $\nearby(A,\delta)$ holds for a suitably chosen $\delta$ and set $A$.
This will yield a $k$-independent site percolation model (see Section~\ref{sec:kindept} for the definition) on $\Tcal$ 
for some suitable $k$. By Lemma~\ref{lem:cross}, for sufficiently large intensities $\lambda$, the probability of sites 
being open will be so small that all open clusters are finite almost surely in this $k$-independent site percolation model.
To conclude the proof, we then observe that if there were to exist an infinite white cluster in the Voronoi percolation model, then all 
the triangles of $\Tcal$ it intersects would be open and hence would form part of some infinite open cluster.

One elementary fact we rely on in our proofs is that a homogeneous Poisson point process on $\Haa^2$ is described by an 
inhomogeneous Poisson process on the unit disk with an intensity function that corresponds to $\lambda$ times the area functional 
of the Poincar\'e disk model.
In some arguments, we switch back and forth between using the Poincar\'e disk metric and the Euclidean metric to generate 
the Voronoi cells.  The reason for doing this is that it allows for comparatively elementary proofs, that
avoid lengthy and/or technical computations.
(Lemma~\ref{lem:adjacency} below shows this change of metric almost 
surely does not change the combinatorial structure of the tessellation. That is, even though the cells will look different, whether 
or not the cell of $z$ meets the cell of $z'$ is unaltered by the change in metric.)

We would like to mention a closely related work by Benjamini and Schramm~\cite{BSconformal} on  
Voronoi percolation on general, smooth, Riemannian manifolds. 
There, Benjamini and Schramm consider crossings (using a different definition 
for crossings from ours) in the situation where one changes the metric in a conformal way, but 
the intensity measure of the underlying Poisson point process remains unchanged, and is comparable to the 
natural area measure of the Riemannian manifold. 
Amongst other things they show that, for any fixed $p$, the large $\lambda$ limit of the crossing 
probabilities -- if the limit exists -- is unchanged by the change in metric. 
In contrast, our Lemma~\ref{lem:cross} establishes that the large $\lambda$ limit 
for the probability of crossing a rectangle is unchanged if $p>1/2$ and we start from the ordinary, Euclidean Voronoi 
percolation model, we leave the metric unchanged, but change the intensity measure to match the area functional of the 
Poincar\'e disk.

\section{Notation and preliminaries\label{sec:preliminaries}}

Here we collect some facts, definitions, notations and results from previous work that we will need for 
the proof of Theorem~\ref{thm:main}.

\subsection{Ingredients from hyperbolic geometry\label{sec:geom}}

The hyperbolic plane $\Haa^2$ is a two dimensional surface with  constant Gaussian curvature $-1$. 
There are many models, i.e. isometric coordinate charts, for $\Haa^2$ including
the Poincar\'e disk model, the Poincar\'e half-plane model, and the Klein disk model.
A gentle introduction to Gaussian curvature, hyperbolic geometry and these representations 
of the hyperbolic plane can be found in~\cite{stillwell2012geometry}.
In this paper we will exclusively use the Poincar\'{e} disk model.  We briefly recollect its definition 
and some of the main facts we shall be using in our arguments below, and refer the reader to~\cite{stillwell2012geometry}
for the proofs and more information. 
The Poincar\'{e} disk model is constructed by equipping the open unit disk $\Dee \subseteq \eR^2$ with an appropriate metric
and area functional.  For points $u,v \in \Dee$ the hyperbolic distance can be given explicitly by 

$$\dist_{\Haa^2}(u,v):=2\arcsinh\left(\frac{\norm{u-v}}{\sqrt{(1-\norm{u}^2)(1-\norm{v}^2)}}\right)$$ 

\noindent
where $\norm{\cdot}$ denotes the Euclidean norm. For any measurable subset $A \subseteq \Dee$ its {\em hyperbolic area} is given by

$$\text{area}_{\Haa^2}(A)=\int_{A}\frac{4}{(1-x^2-y^2)^2}dydx.$$ 

In many of our arguments, we will be simultaneously considering both the Euclidean metric and the hyperbolic metric.
We use subscripts to distinguish the metric under consideration. For instance

$$ \ballR( u, r ) := \{ v \in \eR^2 : \norm{u-v} < r \}, \quad \ballH(u,r) := \{ v \in \Dee : \distH(u,v) < r\}. $$

We refer to $\ballH(u,r)$ as a {\em hyperbolic disk} and to 
$\partial \ballH(u,r) = \{ v \in \Dee : \distH(u,v) = r\}$ as a {\em hyperbolic circle}.
We will make use of the following standard fact. 

\begin{lemma}\label{lem:conformal1} 
Every hyperbolic circle is also a Euclidean circle; and every Euclidean circle contained in $\Dee$ is also a hyperbolic circle.
\end{lemma}

\noindent
(But of course, the centre and radius of a circle with respect to the hyperbolic metric do not coincide with the
centre and radius with respect to the Euclidean metric.)

The geodesic (shortest curve) under the hyperbolic metric between $u,v \in \Dee$ is the segment 
of the unique circle through $u,v$ that intersects $\partial\Dee$ at right angles. (This will be a circle of ``infinite radius'', 
i.e.~a line, if $u,v$ lie on a line through the origin $o$.)
A {\em hyperbolic triangle} refers of course to (the region inside) the three geodesics between three distinct points $a,b,c\in\Dee$.
If $\alpha, \beta, \gamma$ denote the angles at which these geodesics meet, then 
it always holds that $\alpha + \beta + \gamma < \pi$. In fact, for every $\alpha,\beta,\gamma$ for which the inequality holds, there exists a 
triangle $T$ with those angles. 
In particular, there exists an (equilateral) triangle $T$ with all angles equal to $2\pi/7$.
It is possible to tile $\Haa^2$ with copies of this triangle $T$.  Here copy means the image under a $\Haa^2$-isometry, i.e.~a
map that preserves hyperbolic distance and hyperbolic area.
See Section 7.3 of~\cite{stillwell2012geometry} for details and proofs.
The tessellation of $\Haa^2$ by isometric copies of $T$ is the $(7,7,7)$-tesselation the 
notation of Stillwell -- meaning that at every corner of every 
triangle, seven triangles meet.  (See Figure~\ref{fig:order7} for a depiction.)

\begin{figure}[!htb]
\center{\includegraphics[width=.4\textwidth ]{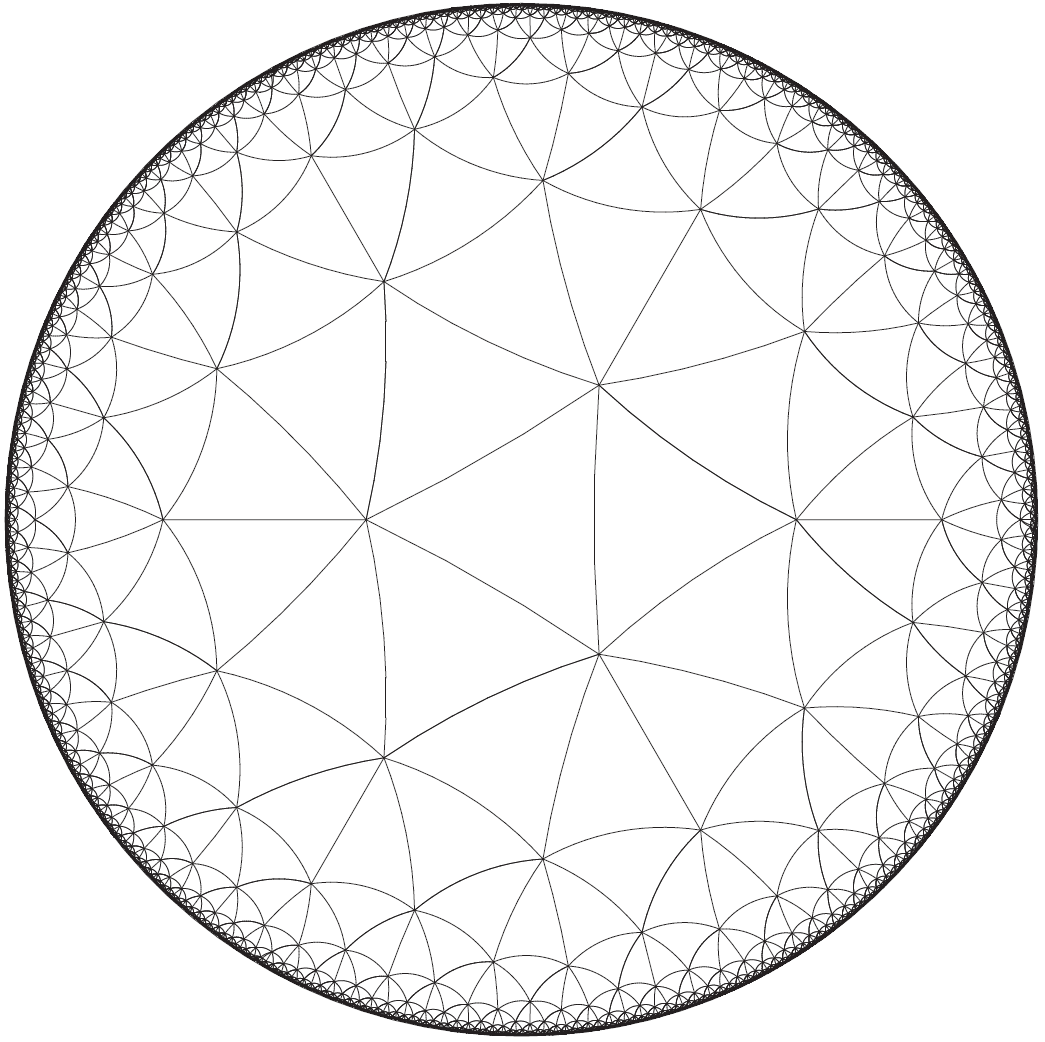} }
\caption{\label{fig:order7} The $(7,7,7)$-triangulation.}
\end{figure}

For $\Zcal$ a countable point set in the hyperbolic, respectively Euclidean, plane we will denote the corresponding 
hyperbolic, respectively Euclidean, Voronoi cells by:

$$ \begin{array}{c} 
\displaystyle \cellH(z;\Zcal) := \{ u \in \Dee : \distH(u,z) = \inf_{z'\in\Zcal} \distH(u,z') \}, \\
\displaystyle \cellR(z;\Zcal) := \{ u \in \eR^2 : \norm{u-z} = \inf_{z'\in\Zcal} \norm{u-z'} \}. 
\end{array} $$

\noindent
Usually the set $\Zcal$ is clear from the context, in which case we will just drop the second argument.

\subsection{Hyperbolic Poisson point processes}

In the rest of this paper $\Zcal$ will usually, but not always, denote a homogeneous Poisson point process (PPP) on $\Haa^2$. 
Analogously to homogeneous Poisson point processes on the ordinary, Euclidean plane, a Poisson process $\Zcal$ of constant intensity 
$\lambda$ on the hyperbolic plane is characterized completely by the properties that
{\bf a)} for each (measurable) set $A\subseteq \Dee$ the random variable $|\Zcal \cap A|$ is Poisson distributed with mean 
$\lambda \cdot \areaH(A)$, and {\bf b)}
if $A_1, \dots, A_m \subseteq \Dee$ are (measurable and) disjoint then the random variables $|\Zcal\cap A_1|, \dots, |\Zcal\cap A_m|$ are 
independent.

In the light of the formula for $\areaH(.)$ above, we can alternatively view $\Zcal$ as an {\em inhomogeneous} Poisson point 
process on the ordinary, Euclidean plane with intensity

$$ u \mapsto \lambda \cdot 1_{\Dee}(u) \cdot \frac{4}{(1-\norm{u}^2)^2}. $$

Throughout the remainder,  we attach to each point of $\Zcal$ a randomly and independently chosen
colour. (Black with probability $p$ and white with probability $1-p$.)
We let $\Zcalb$ denote the black points and $\Zcalw$ the white points of $\Zcal$.
In the language of for instance~\cite{last2017lectures}, we can view $\Zcal$ as a {\em marked} Poisson point process,
the marks corresponding to the colours.
By standard properties of Poisson processes we have that $\Zcalb, \Zcalw$ are {\em independent} Poisson point processes on $\Haa^2$ 
with constant intensities $\lambda \cdot p$, respectively $\lambda \cdot (1-p)$.
We will interchangeably use the point of view of marked point processes and that of a pair of independent Poisson points processes
$(\Zcalb, \Zcalw)$ throughout the paper. 
We use the notation $\Pee_{\lambda,p}(.)$ for the probability measure associated with $\Zcal$ together with its marks, or equivalently
the pair $(\Zcalb,\Zcalw)$. 
In some of our arguments, we are going to want to simultaneously consider a homogeneous Poisson process, with 
marks, on the ordinary, Euclidean plane. In order to keep the two apart, we use $\tilde{\Pee}_{\lambda,p}(.)$ to denote 
the associated probability measure.

\subsection{$k$-independent percolation\label{sec:kindept}}

For a fixed graph $G$, a site percolation probability measure assigns to each vertex the state ``open'' or ``closed''.
Often the situation is considered where the states of the vertices are independent, but for us it will be useful to 
consider a more general situation where there is some dependence.
We say that a measure $\Pee$ on the possible assignments of a state to each vertex
is {\em $k$-independent} if for any set of vertices $S$ such that each pair of distinct vertices in $S$ has graph 
distance $\geq k$, the states of the vertices in $S$ are independent. 
We will make use of the following observation. It is not new.  
A much stronger result is for instance provided by Ligett, Schonmann and Stacey~\cite{liggett1997domination}. 
But, since the proof of the lemma is very short, we choose to give it anyway for the benefit of readers that may not be familiar.

\begin{lemma}\label{lem:kindept}
For every $k, d\in \eN$ there exists a $p_1 = p_1(k, d)>0$ such that the following holds.
For every countable graph $G$ with maximum degree at most $d$ and 
any $k$-independent site percolation measure on $G$ in which each site is open with probability at most $p_1$, we have 

$$\Pee(\text{$\exists$ an infinite open cluster})=0.$$ 

\end{lemma}

\begin{proof}
It is enough to show that $\Pee(\text{$\exists$ an infinite open path starting at $v$}) = 0$ for all vertices $v$.
Since the number of paths of length $\ell$ starting at $v$ is at most $d^\ell$ and each 
path of length $\ell$ contains a set $S$ of size $\geq \ell / (1+d^k)$ with all pairwise distances $\geq k$, we have

$$ \Pee(\text{$\exists$ an open path starting at $v$ of length $\geq \ell$})
\leq d^{\ell} \cdot p_1^{\ell/(1+d^k)} = 
 \left( d \cdot p_1^{1/(1+d^k)} \right)^\ell \xrightarrow[\ell\to\infty]{} 0, $$

\noindent
provided $p_1 < d^{-(1+d^k)}$.
\end{proof}

\subsection{Crossing rectangles in Euclidean Voronoi percolation\label{sec:crosseucl}}

Let  $R \subseteq \eR^2$ be a rectangle that is not a square. 
We say $R$ has a \emph{(long) black crossing} if there is a 
continuous curve $\gamma\subseteq R$ from one short side of $R$ to the other such that all points of $\gamma$ are black
in the colouring of the plane induced by the Voronoi tessellation under the Euclidean metric. We denote this event as $\cross(R)$.
See Figure~\ref{fig:crossing} for a depiction.%
\begin{figure}[!htb]
\center{\includegraphics[width=.7\textwidth ]{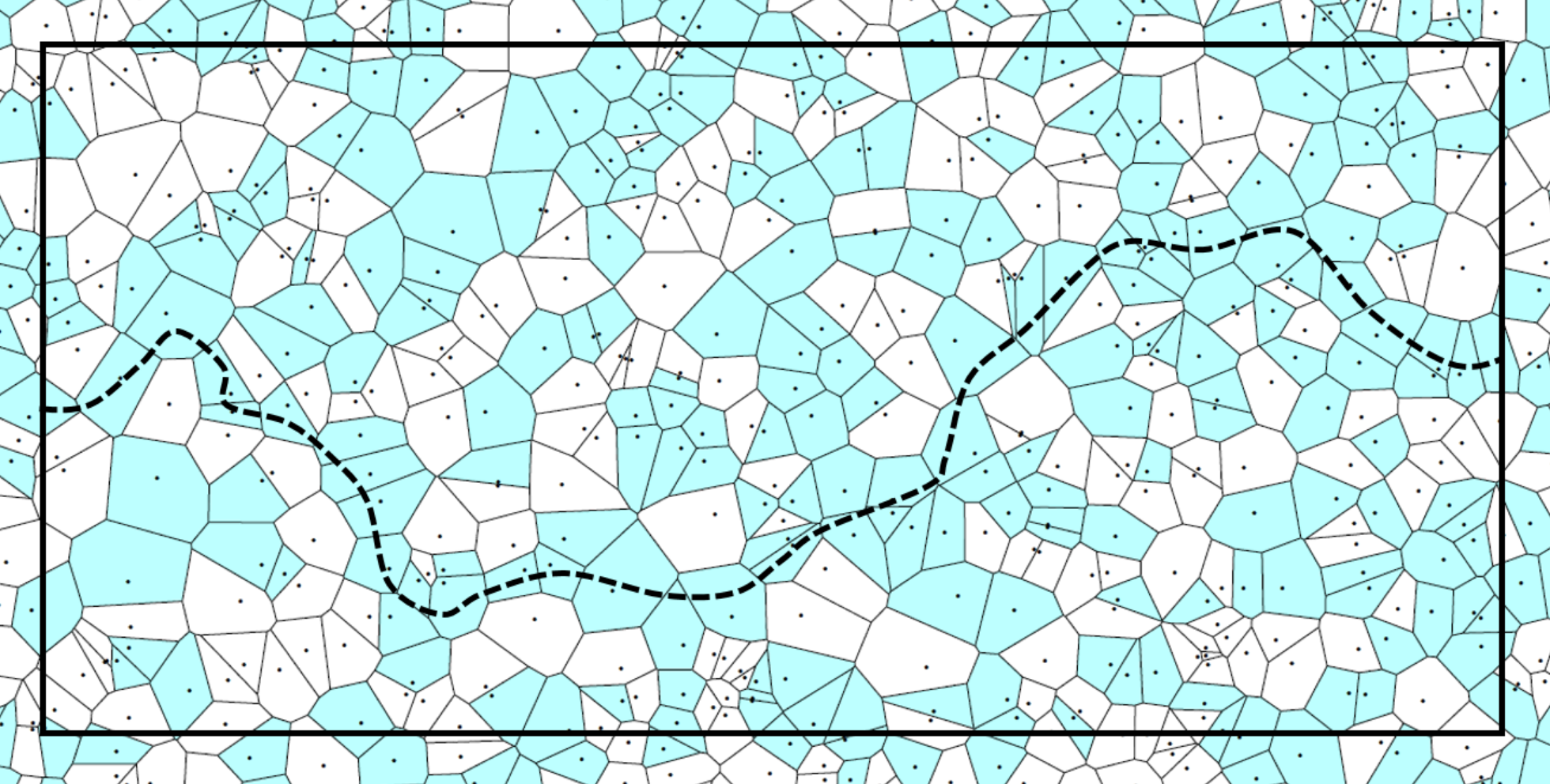}  }
\caption{\label{fig:crossing} A crossing of a rectangle.%
}
\end{figure}
The following result is implicit in the work of Bollob{\'a}s and Riordan~\cite{bollobas2006critical} 
on Voronoi percolation for 
homogeneous Poisson point processes on the ordinary, Euclidean plane $\eR^2$ (and using the Euclidean metric).
Recall that we use $\tilde{\Pee}_{\lambda,p}(.)$ to denote the associated probability measure.

\begin{proposition}\label{prop:bollrior}
For any fixed $p>1/2$ and rectangle $R\subseteq\eR^2$ we have

$$\lim_{\lambda\to\infty}\tilde{\Pee}_{\lambda,p}(\cross(R))=1.$$

\end{proposition}

Since this result does not appear in~\cite{bollobas2006critical} in this precise form, 
we briefly explain how it follows from the results in that paper.
Theorem 7.1 in~\cite{bollobas2006critical} states that for any $p>1/2$ and $\rho>1$  we 
have $\limsup_{s\to\infty} \tilde{\Pee}_{1,p}( \cross( [0,\rho s]\times[0,s]) ) = 1$. That is,
we keep the intensity fixed at $\lambda=1$ and consider larger and larger rectangles with fixed 
aspect ratio $\rho$. By considering the effect of a dilation on a homogeneous, Euclidean Poisson point process
and its Voronoi diagram, this is easily seen to be equivalent to 
$\limsup_{\lambda\to\infty} \tilde{\Pee}_{\lambda,p}(\cross(R) ) = 1$ for any fixed rectangle $R\subseteq \eR^2$. 
That we can replace $\limsup$ by $\lim$ can be seen via a comparison to $1$-independent edge percolation on $\Zed^2$ as
in the proof of Theorem 1.1 in~\cite{bollobas2006critical}. That is, we set up the $1$-independent edge percolation on $\Zed^2$ 
as in that proof, note that the probability that a large rectangle $[0,c \cdot n] \times [0,n]$ is crossed in the $1$-independent
model tends to one, and finally note that this also implies that crossing a rectangle of increasing size and 
fixed aspect ratio tends to one also in the original Voronoi percolation model.

\section{Proofs}

\subsection{The combinatorial structure does not depend on the metric}

We start with an observation that may be of independent interest.
Let $\Zcal$ be a Poisson process on $\Dee$ that has 
constant intensity $\lambda$ with respect to the hyperbolic area measure (or, alternatively, 
intensity $\lambda \cdot 1_{\{u\in\Dee\}} \cdot 4 \cdot (1-\norm{u}^2)^{-2}$ with respect to the ordinary
Lebesgue measure.).
When generating the Voronoi tessellation of $\Zcal$ we could use the hyperbolic metric $\distH$ or the ordinary, Euclidean 
metric.
Of course the tessellations will look different visually (for instance, the sides of the Euclidean Voronoi cells are straight 
line segments, while the sides of the hyperbolic Voronoi cells are circle segments), but if we view the tessellations as 
planar graphs on the vertex set $\Zcal$ with vertices adjacent in the graph if and only if the corresponding Voronoi cells touch 
then in fact the two graphs are (almost surely) identical as the following lemma demonstrates.

\begin{lemma}\label{lem:adjacency}
Almost surely, it holds that $\cellR(z_1)$ and $\cellR(z_2)$ touch if and only if
$\cellH(z_1)$ and $\cellH(z_2)$ touch (for all $z_1, z_2 \in \Zcal$).
\end{lemma}

\begin{proof}
A key elementary observation is that $\cellR(z_1), \cellR(z_2)$ touch if and only if
there exists a disk with $z_1, z_2$ on its boundary and no points of $\Zcal$ in its interior.
(The centres of such disks are precisely the points where the Voronoi cells meet.)
The same is true for $\cellH(z_1), \cellH(z_2)$ except that we need a {\em hyperbolic} disk
with $z_1, z_2$ on its boundary and no points of $\Zcal$ in its interior.
Applying Lemma~\ref{lem:conformal1}, it immediately follows that if $\cellH(z_1)$ and $\cellH(z_2)$ touch, then 
$\cellR(z_1)$ and $\cellR(z_2)$ also touch.
And if $\cellR(z_1), \cellR(z_2)$ touch but $\cellH(z_1), \cellH(z_2)$ do not then 
there exists a (Euclidean) disk with $z_1, z_2$ on its boundary and no other point of $\Zcal$ in its interior, but
{\em every} such disk either ``sticks out'' of the unit disk $\Dee$ or is tangent to its boundary $\partial\Dee$.

Further to the elementary observation above we remark that, almost surely, whenever 
$\cellR(z_1), \cellR(z_2)$ touch there is in fact a third point $z_3 \in \Zcal$ and a disk with 
$z_1, z_2, z_3$ on its boundary and no points of $\Zcal$ in its interior.
(The centre of this disk corresponds to a common corner of $\cellR(z_1), \cellR(z_2), \cellR(z_3)$.)
It is possible that no such $z_3$ exists but in that case all points of $\Zcal$ must be collinear -- which 
almost surely does not happen. 

It thus suffices to show that almost surely there are no triples $z_1, z_2, z_3 \in\Zcal$ such that the unique disk 
$D$ 
with all three points on its boundary {\bf a)} is
either tangent to $\partial\Dee$ or ``sticks out'', yet {\bf b)} has no points of $\Zcal$ in its interior.
We next point out that if {\bf a)} holds then $\areaH( D \cap \Dee ) = \infty$.
(A convenient way to see this without having to integrate is to note that $D$, while not
being a hyperbolic disk itself, contains hyperbolic disks of all radii.)

That there are indeed no triples satisfying {\bf a)} and {\bf b)} almost surely now follows, for instance, using 
the Campbell-Mecke formula (see for instance~\cite[p.30]{last2017lectures}) to show the expected number of 
such triples equals zero.
\end{proof}

\subsection{Locally controlling the Voronoi cells}

For $A \subseteq \eR^2$ we write $A_\delta:= \bigcup_{a\in A} \ballR(a,\delta)$. 
Recall that we use $\Pee_{\lambda,p}(.)$ for the probability measure associated with a homogeneous {\em hyperbolic} Poisson process and
$\tilde{\Pee}_{\lambda,p}(.)$ for a homogeneous {\em Euclidean} Poisson process.

\begin{lemma}\label{lem:nearby}
Let $A\subseteq \Dee$ and $\delta>0$ be such that $A_\delta \subseteq \Dee$.
There exists an event $\nearby(A,\delta)$ with the properties that 
\begin{itemize}
 \item[{\bf (i)}] $\nearby(A,\delta)$ depends only on $\Zcal \cap A_\delta$.
 \item[{\bf (ii)}] We have $\displaystyle \lim_{\lambda\to\infty} \Pee_{\lambda,p}( \nearby(A,\delta) ) 
 = \lim_{\lambda\to\infty} \tilde{\Pee}_{\lambda,p}( \nearby(A,\delta) ) = 1$.
 \item[{\bf (iii)}] If $\nearby(A,\delta)$ holds then, for every $u \in A$,  we have $\displaystyle \inf_{z\in\Zcal}\norm{u-z}<\delta$.
 \item[{\bf (iv)}] If $\nearby(A,\delta)$ holds then, for every $u \in A$, we have 
 $\displaystyle \inf_{z\in\Zcal}\distH(u,z) =\inf_{z\in\Zcal, \atop \norm{u-z}<\delta} \distH(u,z)$.
\end{itemize}
\end{lemma}

Before embarking on the proof of this lemma, let us point out that parts~{\bf(iii)} and {\bf(iv)} tell us that 
if $\nearby(A,\delta)$ holds then the black and white colouring of $A$ produced by the (Euclidean/hyperbolic) Voronoi 
tessellation for $\Zcal$ is completely determined by $\Zcal \cap A_\delta$.

\begin{proof}
We dissect $\eR^2$ into axis-parallel squares of (Euclidean) side length $\delta/1000$ in the obvious way, and
we let $\Scal$ denote the set of those squares in the dissection that are contained in $A_\delta$. 
We define $\nearby(A,\delta)$ as the event that each square in $\Scal$ contains at least one point of $\Zcal$.
Obviously {\bf(i)} holds by construction.

For each square $s\in\Scal$, the number of points that fall in it under the 
Euclidean intensity measure is Poisson distributed with mean 
$\lambda \cdot \areaR(s) = \lambda \cdot \delta^2/10^6 = \Omega(\lambda)$
and under the hyperbolic intensity measure it is Poisson with mean $\lambda \cdot \areaH( s ) = \Omega(\lambda )$.
(Here $\Omega(\lambda)$ denotes a quantity lower bounded by a positive constant times $\lambda$.)
For any fixed square $s\in\Scal$ the probability (under either probability measure) that it contains no point of $\Zcal$ is
thus $\exp[ -\Omega(\lambda) ] \xrightarrow[\lambda\to\infty]{} 0$. 
There are only finitely many squares in $\Scal$, so that part {\bf(ii)} follows by the union bound.

Now pick an arbitrary $u \in A$. 
Since there are (lots of) squares of $\Scal$ contained in $\ballR(u,\delta)$, if $\nearby(A,\delta)$ holds, 
then $u$ is within Euclidean distance $<\delta$ of some $z\in\Zcal$. This takes care of {\bf(iii)}.

To see {\bf(iv)}, let $u \in A$ again be arbitrary and let $z\in \Zcal$ be such that $u \in \cellH(z)$.
This means the hyperbolic disk $D$ with centre $u$ and radius $\distH(u,z)$ has no points of $\Zcal$ in its interior.
The disk $D$ is also a Euclidean disk, but working out its Euclidean centre and radius in terms of $u$ and $z$ 
is a little bit tedious.
What is, however, clear is that $D$ contains the line segment $[u,z]$ with endpoints $u,z$.
So in particular, if $\norm{u-z}\geq \delta$ then $D$ contains a disk $D'$ of Euclidean diameter $\delta/2$ such that $u \in \partial D'$.
Clearly any such disk $D'$ is completely contained in $\ballR(u,\delta) \subseteq A_\delta$ and it contains lots of squares of 
$\Scal$ in its interior. So, if $\nearby(A,\delta)$ holds, then we must have $\norm{u-z} < \delta$.
\end{proof}

\subsection{Coupling the hyperbolic PPP with a Euclidean PPP\label{sec:couple}}

Recall that a {\em coupling} of two random objects $\Xcal, \Ycal$ is a joint probability space for $(\Xcal,\Ycal)$ with the correct 
marginals.
The next lemma allows us to (locally) relate a homogeneous, hyperbolic Poisson process
to a homogeneous, Euclidean Poisson process with different parameters.
This will be instrumental in proving an analogue of Proposition~\ref{prop:bollrior} in the next section.

\begin{lemma}\label{lem:coupling}
Let $0<\pnew<p<1$ and $0<r<1$ be arbitrary. 
There exists $t=t(r,p,p_{new})>0$ such that 
for each (measurable) $A\subseteq \ballR(o,r)$ with Euclidean diameter at most $t$ and each $\lambda>0$ there
exist a $\mu = \mu_{\lambda, A}$ and a coupling of $(\Zcalb, \Zcalw)$ and $(\Zcalbtil, \Zcalwtil)$ satisfying (almost surely)

$$ \Zcalbtil\cap A \subseteq \Zcalb\cap A, \quad \Zcalw\cap A \subseteq \Zcalwtil\cap A, $$

\noindent
where $\Zcalbtil, \Zcalwtil$ are independent, homogeneous, Euclidean Poisson processes of intensities 
$\pnew \cdot \mu$ and $(1-\pnew)\cdot\mu$. 
\end{lemma}

To aid the reader's understanding, we emphasize that $\Zcalb$ is with respect to the hyperbolic intensity measure 
$u \mapsto \lambda \cdot p \cdot 1_{\Dee}(u) \cdot 4 (1-\norm{u}^2)^{-2}$ 
while $\Zcalbtil$ is with respect to the Euclidean intensity measure $u \mapsto \mu \cdot \pnew$, and similarly for 
$\Zcalw$ and $\Zcalwtil$.

Let us also remark that (provided $A$ is not a set of measure zero) we have $\mu_{\lambda,A} \to \infty$ as $\lambda\to\infty$.
This is because the expected number of points of $\Zcalwtil$ that fall in $A$ is at least the expected number
of points $\Zcalw$ that fall in $A$, giving $(1-\pnew) \cdot \mu \cdot \areaR(A) \geq (1-p) \cdot \lambda \cdot \areaH(A)$.

\begin{proof}
Let $\delta>0$ be small, to be specified appropriately later on in the proof.
Since $f(u) := 1_{\Dee}(u) \cdot 4 (1-\norm{u}^2)^{-2}$ is uniformly continuous on 
$\ballR(o,r)$ we can and do choose $t = t(r,p,\pnew)$ such that 
$|f(u)-f(v)| < \delta$ for all $u,v \in \ballR(o,r)$ with $\norm{u-v}<t$.
Now let $A \subseteq \ballR(0,r)$ be an arbitrary measurable set of Euclidean diameter $\leq t$.
We set $\mu := \lambda \cdot \inf_{u\in A} f(u)$.
Since $\Zcalb, \Zcalw$ are independent Poisson processes, and similarly for $\Zcalbtil, \Zcalwtil$, it 
suffices to construct a coupling of $\Zcalb$ with $\Zcalbtil$ and a coupling of $\Zcalw$ with $\Zcalwtil$ separately.

In order to construct the coupling of $\Zcalb, \Zcalbtil$ we first note that with our choice of $\mu$ we have 
$\lambda \cdot p \cdot f(u) \geq \mu \cdot \pnew$ for all $u \in A$. 
We let $\Pcal_0, \Pcal_1, \Pcal_2$ be independent, inhomogeneous Poisson processes
with intensities given by 

$$ \begin{array}{c} 
\varphi_0(u) := \min\left( \pnew \cdot \mu,\ p \cdot \lambda \cdot f(u) \right), \\[1ex]
\varphi_1(u) :=   \pnew \cdot \mu - \varphi_0(u), \quad \quad 
\varphi_2(u) := \ p \cdot \lambda \cdot f(u) - \varphi_0(u).
\end{array} $$

\noindent
As $\varphi_1$ is identically zero on $A$ we have

$$(\Pcal_0 \cup \Pcal_1) \cap A = \Pcal_0 \cap A \subseteq (\Pcal_0\cup\Pcal_2)\cap A \quad (\text{almost surely}). $$

\noindent
Since $\varphi_0 + \varphi_1 = \pnew \cdot \mu$ and $\varphi_0 + \varphi_2 = p \cdot \lambda \cdot f$, 
the superposition theorem for Poisson point 
processes (see e.g.~\cite[p. 20]{last2017lectures}) implies that $\Zcalb \isd \Pcal_0 \cup \Pcal_1$ and 
$\Zcalbtil \isd \Pcal_0 \cup \Pcal_2$. In other words, the pair $\Pcal_0\cup\Pcal_1,\Pcal_0\cup\Pcal_2$ provides a 
coupling of $\Zcalb, \Zcalbtil$ with the desired properties.

In order to construct the coupling of $\Zcalw, \Zcalwtil$ we remark that, for all $v \in A$, we have
$f(v) \leq \inf_{u\in A} f(u) + \delta \leq (1+\delta) \cdot \inf_{u\in A} f(u)$. 
Having chosen $\delta=\delta(p,\pnew)$ sufficiently small, we have, for every $v\in A$:

$$ \begin{array}{rcl} 
(1-p) \cdot \lambda \cdot f(v) 
& \leq &  (1+\delta) \cdot (1-p) \cdot \lambda \cdot \inf_{u\in A} f(u) \\
& \leq & (1-\pnew) \cdot \lambda \cdot \inf_{u\in A} f(u) \\ 
& = & (1-\pnew) \cdot \mu. 
\end{array} $$

\noindent
We can, therefore, construct the sought coupling of $\Zcalw, \Zcalwtil$ analogously to before.
\end{proof}

\subsection{Crossing rectangles with the Euclidean metric and hyperbolic intensity measure}

We are now ready to prove the following analogue of Proposition~\ref{prop:bollrior} 
for the hyperbolic intensity measure.
All mention of rectangles in the remainder of this paper will be with respect to the Euclidean metric. 
(In fact, it is not immediately clear what would be the best notion of a ``rectangle'' in hyperbolic geometry.)
In other words, for us a rectangle is a subset $A$ of the hyperbolic plane such that, when we display the hyperbolic plane
as the unit disk in the Euclidean plane using the Poicar\'e disk model, $A$ is also a Euclidean rectangle.
Or, put differently, a rectangle is a subset of the hyperbolic plane that {\em looks like} a Euclidean rectangle (that is not a 
square) in the Poincar\'e disk model.
We emphasize that in the definition of the events $\cross(R)$ we ask for an all black, continuous curve inside $R$
that connects the two shorter sides, where black refers to the colouring generated by the Voronoi tessellation using the
{\em Euclidean} metric.

\begin{lemma}\label{lem:cross}
For any fixed $p>1/2 $ and any fixed rectangle $R\subseteq \Dee$ we have

$$ \lim_{\lambda\to\infty}\Pee_{\lambda,p}(\cross(R))=1.  $$

\end{lemma}

\begin{figure}[!htb]
\center{\includegraphics[width=.8\textwidth ]{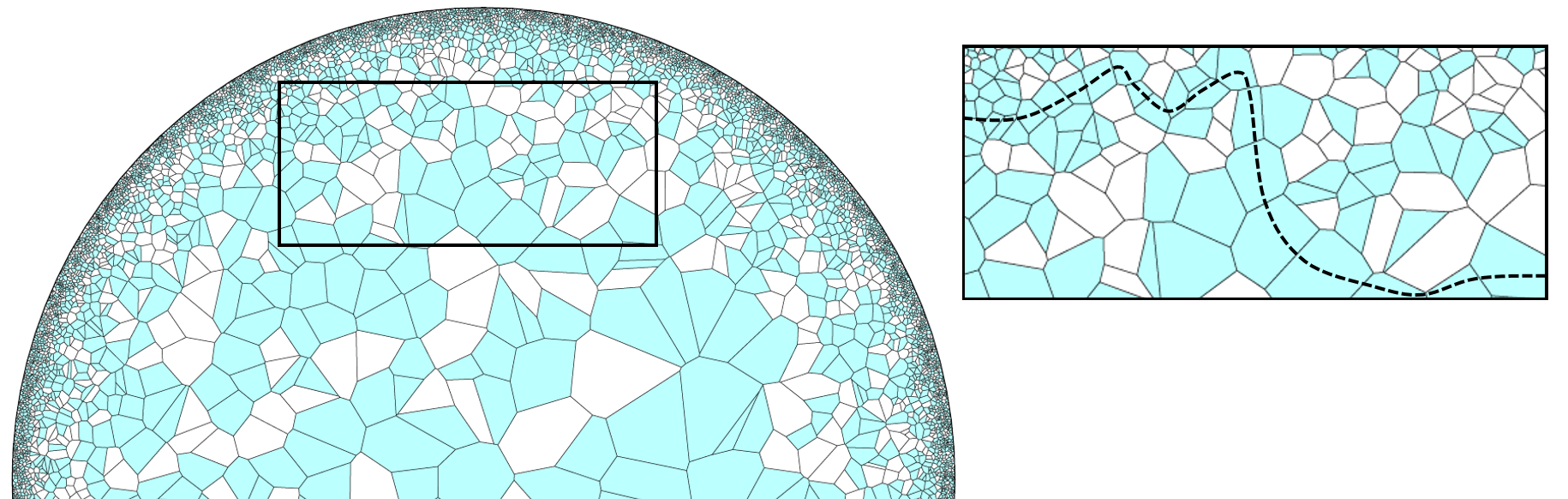} }
\caption{\label{fig:rectperc} The Euclidean  Voronoi cells under hyperbolic intensity measure 
and a rectangle $R$ for which $\cross(R)$ occurs. ($p=.6, \lambda=20$)%
}
\end{figure}

\begin{proof}
We first point out that it is sufficient to show that for each $0<r<1$ the statement holds 
for rectangles contained in $\ballR(o,r)$, with (Euclidean) diameter at most some small constant $\ell=\ell(r,p)$ (to be
chosen appropriately in the course of the proof). 
This is because for any fixed rectangle $R \subseteq \Dee$ and $r=r(R)$ sufficiently close to one and any $\ell > 0$, we can 
place a finite number of rectangles $R_1, \dots, R_m$, each of Euclidean diameter at most $\ell$ and contained in 
$\ballR(o,r)$, such that the event $\cross(R_1) \cap \dots \cap \cross(R_m)$ implies the event $\cross(R)$.
(See e.g.~Figure~\ref{fig:hrectangles}.)

\begin{figure}[!htb]
\center{\includegraphics[width=.9\textwidth ]{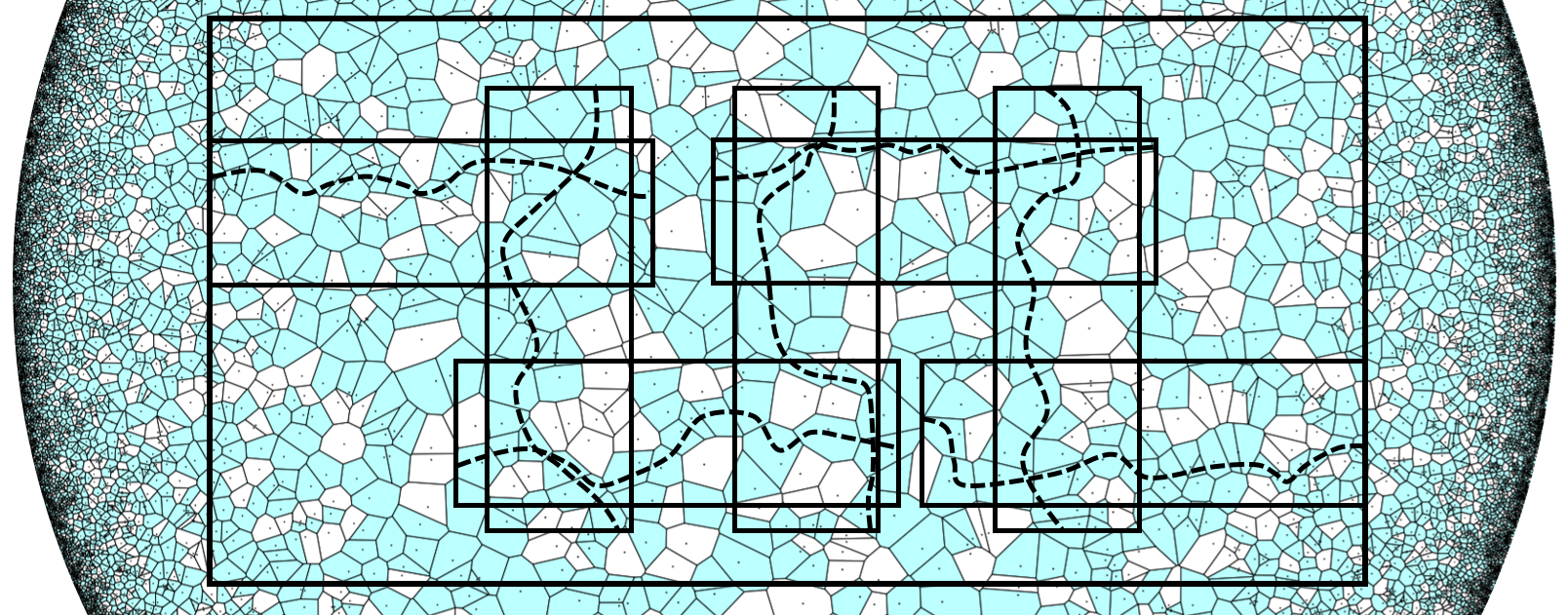}  }
\caption{\label{fig:hrectangles}  Crossing a big rectangle using smaller ones.}
\end{figure}

\noindent
So $\Pee_{\lambda,p}(\cross(R) ) \geq 1 - \sum_{i=1}^m \Pee_{\lambda,p}(\cross(R_i)^c )$ and it is enough 
to show that $\Pee_{\lambda,p}(\cross(R_i) ) \to 1$ as $\lambda\to\infty$, for all $i=1,\dots,m$.

We thus fix $0<r<1$ and we let $R \subseteq \ballR(o,r)$ be an arbitrary rectangle of (Euclidean) diameter 
at most $\ell := t/3$ where $t = t((1+r)/2,p,(1/2+p)/2)$ is as provided by Lemma~\ref{lem:coupling}.
For convenience, we will write $\rho := (1+r)/2, \pnew = (1/2+p)/2$ from now on.
We apply Lemma~\ref{lem:coupling} to $R_\delta$ where $\delta := \min(t/3,\rho-r)$ is such that 
$\diamR( R_\delta ) \leq t$ and $R_\delta \subseteq \ballR(o,\rho)$ and we 
let $\mu=\mu_{\lambda,R_\delta}$ and $\Zcalbtil, \Zcalwtil$ be as provided by that lemma.
We denote by $B \subseteq R$ the (random) subset of $R$ that is coloured black in the Voronoi tessellation 
for $\Zcalb, \Zcalw$ and we denote by $\tilde{B}\subseteq R$ the black subset of  $R$ under the Voronoi tessellation
for $\Zcalbtil, \Zcalwtil$. (In both cases the Voronoi tessellation is with respect to the Euclidean metric.) 

Suppose for a moment that $\nearby(R,\delta)$ holds both for $\Zcal = \Zcalb \cup \Zcalw$
and for $\Zcaltil = \Zcalbtil \cup \Zcalwtil$.
In that case, by the remark immediately following Lemma~\ref{lem:nearby},  
$B$ is completely determined by $(\Zcalb\cap R_\delta, \Zcalw\cap R_\delta)$ and $\tilde{B}$ is completely determined by 
$(\Zcalbtil\cap R_\delta, \Zcalwtil\cap R_\delta)$.
Since under our coupling $\Zcalbtil\cap R_\delta \subseteq \Zcalb\cap R_\delta$ and 
$\Zcalwtil \cap R_\delta \supseteq \Zcalw\cap R_\delta$, this would give that $B \supseteq \tilde{B}$. 
In particular,  if $\tilde{B}$ contains a crossing of $R$ then so does $B$ (still under the assumption that 
$\nearby(R,\delta)$ happens both for $\Zcal$ and $\Zcaltil$).

We may conclude that
$$ \Pee_{\lambda,p}(\cross(R)) \geq \tilde{\Pee}_{\mu,\pnew}(\cross(R)) 
- \Pee_{\lambda,p}(\nearby(R,\delta)^c ) - \tilde{\Pee}_{\mu,\pnew}(\nearby(R,\delta)^c).
$$

By Proposition~\ref{prop:bollrior}, and the remark following Lemma~\ref{lem:coupling} 
(stating that $\mu\to\infty$ as $\lambda\to\infty$),
we have $\lim_{\lambda\to\infty} \tilde{\Pee}_{\mu,\pnew}(\cross(R)) = 1$.
By part {\bf(ii)} of Lemma~\ref{lem:nearby}, we have $\displaystyle\lim_{\lambda\to\infty} \Pee_{\lambda,p}(\nearby(R,\delta) ) =
\lim_{\lambda\to\infty} \tilde{\Pee}_{\mu,\pnew}(\nearby(R,\delta) ) = 1$.
The result follows.
\end{proof}

\subsection{Proof of our main result}

For ease of notation, we will say that $z, z' \in \Zcal$ are {\em adjacent} if the corresponding
Voronoi cells $\cellH(z), \cellH(z')$ touch. In light of Lemma~\ref{lem:adjacency}, this 
is equivalent to $\cellR(z), \cellR(z')$ touching, up to an event of probability zero.
A path is of course a (finite or infinite) sequence of distinct points $z_1, z_2, \dots \in \Zcal$ such that $z_i, z_{i+1}$ are adjacent
for each $i$.

\begin{proofof}{Theorem~\ref{thm:main}} 
As explained in the introduction, it suffices to show that for every $p>1/2$ we have
$\Pee_{\lambda,p}( \text{$\exists$ infinite white component} ) = 0$ for all sufficiently large $\lambda$. %
In order to show this, we will define a dependent percolation model on the $(7,7,7)$-triangulation $\Tcal$ of $\Haa^2$. 
We turn $\Tcal$ into a graph by declaring two triangles adjacent if and only if they meet in at least one point. 
So all vertices of the graph have degree 15.
The state (open/closed) of each triangle $T\in\Tcal$ will be determined by the Poisson-Voronoi tessellation, in such a 
way that the existence of an infinite white component in the Poisson-Voronoi tessellation implies the existence
of an infinite open cluster in the dependent percolation model.
For convenience, we assume without loss of generality that one of the triangles $T_o \in \Tcal$ is centred at the origin. 

We start by defining the event $\closed(T_o)$ that $T_o$ is closed. We fix six thin rectangles $R_1, \dots, R_6$ 
as pictured in Figure~\ref{fig:order7_rect}. 
The key features of this placement are that each rectangle is at least some positive distance away from
both $T_o$ and $\partial\Dee$, and that the event $\cross(R_1)\cap\dots\cap \cross(R_6)$ will imply the existence
of a black, continuous, closed curve that separates $T_o$ from $\partial\Dee$.
A subtle point here is that, because of the way we've defined the events $\cross(.)$, this black curve separating 
$T_o$ from $\partial\Dee$ is with respect to the colouring of $\Dee$ generated
by the Voronoi tessellation under the Euclidean metric.

\begin{figure}[!htb]
\center{ \includegraphics[width=.4\textwidth ]{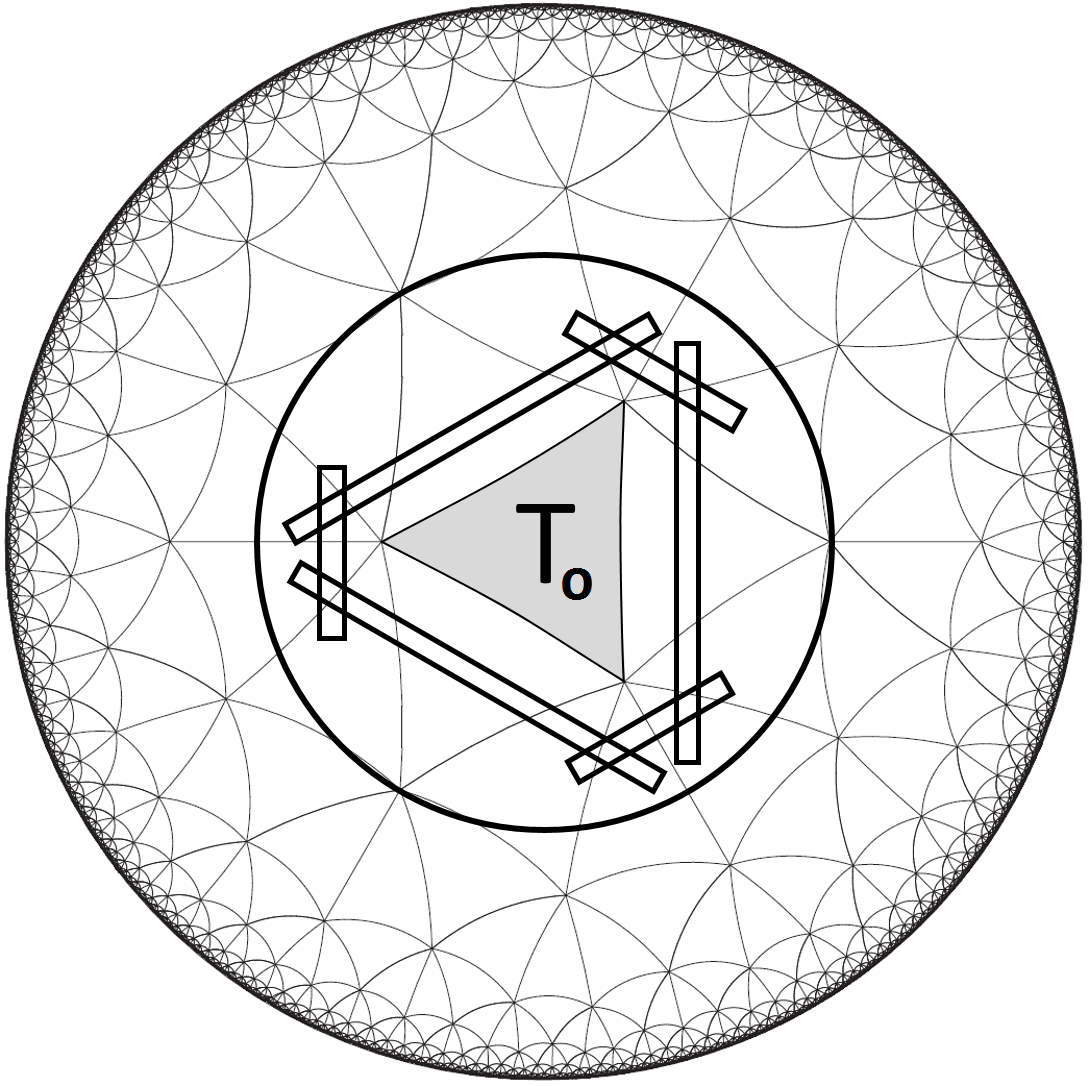} }
\caption{\label{fig:order7_rect} Six rectangles inside $\ballR(o,r)$, surrounding $T_o$.}
\end{figure}

We can and do pick an $0<r<1$ and  $0<\delta<1-r$ such that $R_1, \dots, R_6 \subseteq \ballR(o,r)$ 
and each $R_i$ has Euclidean distance $>\delta$ to $T_o$.
Now we define

$$ \closed(T_o) := \cross(R_1) \cap \dots \cap \cross(R_6) \cap \nearby( \ballR(o,r), \delta). $$

For each triangle $T \in \Tcal$ in the tiling, we fix a $\Haa^2$-isometry $\varphi$ that maps $T$ to $T_o$ and 
we define $\closed(T)$ as the event that $\closed(T_o)$ holds with respect to $\varphi[\Zcal]$.
Of course $T \in \Tcal$ is declared open if $\closed(T)$ does not hold.
And, since $\varphi[\Zcal]$ has the same distribution as $\Zcal$, we have

\begin{equation}\label{eq:lim} 
\Pee_{\lambda,p}( \closed(T) ) = \Pee_{\lambda,p}( \closed(T_o) ) \xrightarrow[\lambda\to\infty]{} 1, 
\end{equation}

\noindent
the limit holding because of Lemmas~\ref{lem:nearby} and~\ref{lem:cross}.
Moreover, by part {\bf(i)} of Lemma~\ref{lem:nearby} and the remark following that lemma, the event 
$\closed(T_o)$ depends only on the points of $\Zcal$ that fall inside $\ballR(o,r+\delta)$.
Of course, $\ballR(o,r+\delta)$ is also a hyperbolic disk $\ballH( o, \rho )$ centred at the origin for some 
finite $\rho = \rho(r+\delta)$. 
By construction of the events $\closed(T)$, for any other $T \in \Tcal$, the event $\closed(T)$ 
also depends only on the part of the Poisson process $\Zcal$ inside a hyperbolic disk of radius $\rho$ around the centre of $T$.
The percolation model we've defined on $\Tcal$ is, therefore, $k$-independent where 
$k$ is the number of triangles of $\Tcal$ whose centre has hyperbolic distance $<2\rho$ to the origin.
Combining Lemma~\ref{lem:kindept} with~\eqref{eq:lim} we find that 
for sufficiently large $\lambda$:

$$ \Pee_{\lambda,p}( \text{$\exists$ an infinite open cluster in $\Tcal$} ) = 0. $$

It remains to see how this implies that the probability that an infinite white cluster exists in the 
Voronoi tessellation is also zero for large values of the intensity $\lambda$. 
A key observation is the following.

\vspace{1ex}

\begin{quote}
{\bf Claim.} 
{\em %
Almost surely, for every $T \in \Tcal$, if $\closed(T)$ holds then there does not exist an infinite white path $z_1, z_2,\dots$
such that $\bigcup_{i=1}^\infty \cellH(z_i)$ intersects $T$.
} 

\vspace{1ex}

\begin{proofof}{the claim} 
We first point out that it suffices to prove the claim for $T=T_o$.
This is because if $\varphi : \Dee \to \Dee$ is the $\Haa^2$-isometry mapping $T$ to $T_o$ used in the definition of
$\closed(T)$, then $\varphi[ \cellH(z;\Zcal) ] = \cellH(\varphi(z),\varphi[\Zcal])$ for all $z\in\Zcal$.
So $z,z'$ are adjacent in the Voronoi tessellation for $\Zcal$ if and only if $\varphi(z), \varphi(z')$ are adjacent in the 
Voronoi tessellation for $\varphi[\Zcal]$ and $\cellH(z;\Zcal)$ intersects $T$ if and only if
$\cellH(\varphi(z);\varphi[\Zcal])$ intersects $T_o$.

In the remainder, we shall thus be taking $T=T_o$.
The occurrence of $\cross(R_1)\cap\dots\cap\cross(R_6)$ implies that in the 
colouring generated by the Voronoi tessellation using the {\em Euclidean} metric there is an all black, continuous, closed curve $\gamma
\subseteq R_1\cup\dots\cup R_6$ that separates $T_o$ from $\partial \ballR(0,r)$. 

Suppose there is an infinite white path $z_1, z_2,\dots$ as in the statement of the claim, and let $i$ be such that
$\cellH(z_i) \cap T_o \neq \emptyset$.
Since $\ballR(o,r)$ contains finitely many points almost surely, there will be some $j>i$ such that
$z_j \not\in \ballR(o,r)$.
The occurrence of $\nearby( \ballR(o,r), \delta )$ implies that $z_i$ is within Euclidean distance $\delta$ of $T_o$.  
In particular, the black curve $\gamma$ separates $z_i$ and $z_j$. 
Since $\cellR(z_i) \cup \cellR(z_{i+1}) \cup \dots \cup \cellR(z_j)$ contains a continuous curve between $z_i$ and $z_j$,
at least one of the Euclidean cells $\cellR(z_i), \cellR(z_{i+1}), \dots, \cellR(z_j)$ intersects the black curve $\gamma$. 
But that means one of $z_i, z_{i+1}, \dots, z_j$ must be black, contradicting our choice of $z_1, z_2, \dots$.
\end{proofof}
\end{quote}

To conclude the proof we remark that if there were to exist an infinite white path $z_1, z_2, \dots$
in the hyperbolic Voronoi diagram, then the set of triangles of $\Tcal$ that $\cellH(z_1) \cup \cellH(z_2) \cup \dots$ intersects
would have to be part of an infinite open cluster in $\Tcal$.
\end{proofof}

\subsection*{Acknowledgement} 

We thank the anonymous referees for their useful comments.

\bibliographystyle{amsplain}
\bibliography{voronoibib}

\end{document}